\documentclass[12pt]{amsart}
\usepackage{hyperref}
\usepackage{amsmath}
\usepackage{amssymb}
\usepackage{amsthm}
\usepackage{enumerate}
\usepackage{enumerate}
\usepackage[margin=2.5cm]{geometry}

\usepackage{tikz}
\usepackage{tikz-cd}
\usetikzlibrary{patterns,snakes}
\newcommand\abs[1]{\lvert #1\rvert}

\newenvironment{subproof}[1][\proofname]{%
  \begin{proof}[#1]%
}{%
  \end{proof}%
}

\def\COMMENT#1{}
\let\COMMENT=\footnote%

\newtheorem{theorem}{Theorem}[section]
\newtheorem{lemma}[theorem]{Lemma}
\newtheorem{proposition}[theorem]{Proposition}
\newtheorem{corollary}[theorem]{Corollary}
\newtheorem{conjecture}{Conjecture}[section]
\newtheorem{problem}[conjecture]{Problem}
\newtheorem{claim}{Claim}

\newtheorem*{definition*}{Definition}

\newcommand\pivot\wedge

\renewcommand{\epsilon}{\varepsilon}

\begin{document}

\title{The Erd{\H{o}}s-Hajnal 
property for graphs with no fixed cycle as a pivot-minor}

\author[J. Kim]{Jaehoon Kim}
\address[J. Kim]{Department of Mathematical Sciences,  KAIST, Daejeon, South Korea}
\email{jaehoon.kim@kaist.ac.kr}
\author[S. Oum]{Sang-il Oum}
\address[S. Oum]{Discrete Mathematics Group, Institute for Basic Science (IBS), Daejeon, South Korea}
\address[S. Oum]{Department of Mathematical Sciences,  KAIST, Daejeon, South Korea}
\email{sangil@ibs.re.kr}
\thanks{J.K.\ was supported by the POSCO Science Fellowship of POSCO TJ Park Foundation and by the KAIX Challenge program of KAIST Advanced Institute for Science-X (J.~Kim). S.O.\ was supported by the Institute for Basic Science (IBS-R029-C1).}

\begin{abstract}
    We prove that for every integer $k$, there exists $\varepsilon > 0$ such that for every n-vertex graph $G$ with no pivot-minor isomorphic to $C_k$, there exist disjoint sets $A,B \subseteq V(G)$ such that $|A|,|B| \geq \varepsilon n$, and $A$ is either complete or anticomplete to $B$.
This proves the analog of the Erd\H{o}s-Hajnal conjecture for 
the class of graphs with no pivot-minor isomorphic to $C_k$.
\end{abstract}

\date{\today}
\maketitle

\section{Introduction}
In this paper all graphs are simple, having no loops and no parallel edges.
For a graph~$G$, let $\omega(G)$ be the maximum size of a clique, that is 
a set of pairwise adjacent vertices
and let $\alpha(G)$ be the maximum size of an independent set, that is 
a set of pairwise non-adjacent vertices.
Erd\H{o}s and Hajnal~\cite{EH1989} proposed the following conjecture in 1989.
\begin{conjecture}[Erd\H{o}s and Hajnal~\cite{EH1989}]\label{con:eh}
For every graph $H$, 
there is $\varepsilon>0$ such that all graphs $G$ with no induced subgraph isomorphic to $H$
satisfies 
\[\max(\alpha(G),\omega(G))\ge \abs{V(G)}^\varepsilon.\]
\end{conjecture}
This conjecture still remains open. 
See~\cite{Chudnovsky2013} for a survey on this conjecture.
We can ask the same question for weaker containment relations. 
Recently Chudnovsky and Oum~\cite{CO2018} proved that this conjecture holds if we replace ``induced subgraphs'' with ``vertex-minors'' as follows. This is weaker 
in the sense that every induced subgraph $G$ is a vertex-minor of $G$
but not every vertex-minor of $G$ is an induced subgraph of $G$.
\begin{theorem}[Chudnovsky and Oum~\cite{CO2018}]\label{thm:vertexminor}
For every graph $H$, there exists $\varepsilon>0$ such that 
every graph with no vertex-minors isomorphic to $H$ satisfies 
\[\max(\alpha(G),\omega(G))\ge \abs{V(G)}^\varepsilon.\]
\end{theorem}

We ask whether Conjecture~\ref{con:eh} holds if we replace ``induced subgraphs'' 
with ``pivot-minors'' as follows.
\begin{conjecture}\label{con:pivot}
For every graph $H$, there exists $\varepsilon>0$ such that 
every graph $G$ with no pivot-minor isomorphic to $H$
satisfies 
\[\max(\alpha(G),\omega(G))\ge \abs{V(G)}^\varepsilon.\]
\end{conjecture}
The detailed definition of pivot-minors will be presented in Section~\ref{sec:cycle}.
For now, note that the analog for vertex-minors is weakest,
the analog for pivot-minors is weaker than that for induced subgraphs
but stronger than that for vertex-minors. 
This is because every induced subgraph of $G$ is a pivot-minor of $G$,
and every pivot-minor of $G$ is a vertex-minor of $G$.
In other words, Conjecture~\ref{con:eh} implies Conjecture~\ref{con:pivot}
and Conjecture~\ref{con:pivot} implies Theorem~\ref{thm:vertexminor}. We verify Conjecture~\ref{con:pivot} for  $H=C_k$, the cycle graph on $k$ vertices as follows.
\begin{theorem}\label{thm:eh}
For every $k\ge 3$, there exists $\varepsilon>0$ such that 
every graph with no pivot-minor isomorphic to $C_k$
satisfies 
\[\max(\alpha(G),\omega(G))\ge \abs{V(G)}^\varepsilon.\]
\end{theorem}

We actually prove a stronger property, as Chudnovsky and Oum~\cite{CO2018} did. 
Before stating this property, let us first state a few terminologies.
A class $\mathcal G$ of graphs closed under taking induced subgraphs is said to have \emph{the Erd\H{o}s-Hajnal property}
if there exists $\varepsilon>0$ such that every graph $G$ in $\mathcal{G}$
satisfies \[\max(\alpha(G),\omega(G))\ge \abs{V(G)}^\varepsilon.\]
A class $\mathcal G$ of graphs closed under taking induced subgraphs is said to 
have \emph{the strong Erd\H{o}s-Hajnal property} 
if there exists $\varepsilon>0$ such that 
every $n$-vertex graph in $\mathcal G$ with $n>1$ has disjoint sets $A$, $B$ of vertices such that $|A|,|B|\geq \varepsilon n$ and $A$ is either complete or anti-complete to $B$.
It is an easy exercise to show that 
the strong Erd\H{o}s-Hajnal property implies the Erd\H{o}s-Hajnal property,
see~\cite{APPRS2005,FP2008a}.

Chudnovsky and Oum~\cite{CO2018} proved that the class of graphs with no vertex-minors isomorphic to $H$ for a fixed graph $H$
has the strong Erd\H{o}s-Hajnal property, implying Theorem~\ref{thm:vertexminor}.
We propose its analog for pivot-minors as a conjecture, which implies the theorem of Chudnovsky and Oum~\cite{CO2018}. Note that this conjecture is not true if we replace the pivot-minor with induced graphs. For example, the class of triangle-free graphs does not have the strong Erd\H{o}s-Hajnal property~\cite{FP2008a}.
\begin{conjecture}\label{con:strongeh}
    For every graph $H$, there exists $\varepsilon>0$ such that 
    for all $n>1$, 
    every $n$-vertex graph with no pivot-minor isomorphic to $H$
    has two disjoint sets $A$, $B$ of vertices 
    such that
    $\abs{A},\abs{B}\ge \varepsilon n$
    and 
    $A$ is complete or anti-complete to $B$.
\end{conjecture}
We prove that this conjecture holds if $H=C_k$. In other words, 
the class of graphs with no pivot-minor isomorphic to $C_k$
has the strong Erd\H{o}s-Hajnal property as follows. This implies Theorem~\ref{thm:eh}.

\begin{theorem}\label{thm: Strong EH}
For every integer $k\ge 3$, there exists $\varepsilon>0$ such that 
for all $n>1$, 
every $n$-vertex graph with no pivot-minor isomorphic to $C_k$
has two disjoint sets $A$, $B$ of vertices 
such that
$\abs{A},\abs{B}\ge \varepsilon n$
and 
$A$ is complete or anti-complete to $B$.
\end{theorem}

This paper is organized as follows.
In Section~\ref{sec:prelim}, we will introduce basic definitions
and review necessary theorems of R\"odl~\cite{Rodl1986}
and Bonamy, Bousquet, and Thomass\'e~\cite{BBT}.
In Section~\ref{sec:cycle}, we will present several tools to find a pivot-minor isomorphic to $C_k$. In particular, it proves that a long anti-hole contains $C_k$ as a pivot-minor.
In Section~\ref{sec:proof}, we will present the proof of the main theorem, Theorem~\ref{thm: Strong EH}.
In Section~\ref{sec:discussion}, 
we will relate our theorem to 
the problem on $\chi$-boundedness, and discuss known results and open problems related to polynomial $\chi$-boundedness and the Erd\H{o}s-Hajnal property.

\section{Preliminaries}\label{sec:prelim}

Let $\mathbb{N}$ be the set of positive integers and for each $n\in \mathbb{N}$, we write $[n]:=\{1,2,\dots,n\}$.
For a graph $G=(V,E)$, let $\overline{G}= (V, \binom{V}{2}- E)$ be the complement of $G$. 
We write $\Delta(G)$ and $\delta(G)$
to denote the maximum degree of $G$ and the minimum degree of $G$ respectively.

Let $T$ be a tree rooted at a specified node $v_r$, called the \emph{root}. If the path from  $v_r$ to a node $y$ in $T$ contains $x\in V(T)-\{y\}$, we say that $x$ is an \emph{ancestor} of $y$, and $y$ is a \emph{descendant} of $x$. If one of $x$ and $y$ is an ancestor of the other, we say that $x$, $y$ are \emph{related}. We say that two disjoint sets $X$ and $Y$ of nodes of $T$ are \emph{unrelated} if no pairs of $x\in X$ and $y\in Y$ are related.

For disjoint vertex sets $X$ and $Y$, we say $X$ is \emph{complete} to $Y$ if every vertex
of $X$ is adjacent to all vertices of $Y$. 
We say $X$ is \emph{anti-complete} to $Y$
if every vertex of $X$ is non-adjacent to~$Y$.
A \emph{pure pair} of a graph $G$ is
a pair $(A,B)$ of disjoint subsets of $V(G)$
such that $A$ is complete or anticomplete to $B$.

For a vertex $u$, let $N_{G}(u)$ denote the set of neighbors of $u$ in $G$.
For each $U\subseteq V(G)$, we write 
\[
    N_G(U)  :=\bigcup_{u\in U}N_G(u)- U.
\]

The following lemma is proved in Section~2 of~\cite{BBT}.
\begin{lemma}[Bonamy, Bousquet, and Thomass\'e~\cite{BBT}]\label{lem: tree}
For every connected graph $G$ and a vertex $v_r\in V(G)$, there exist an induced subtree $T$ of $G$ rooted at $v_r$ and a function $r: V(G)\rightarrow V(T)$ satisfying the following.
\begin{enumerate}[(T1)]
\item\label{T1} $r(v_r)=v_r$ and for each $u\in V(G) - \{v_r\}$, the vertex $r(u)$ is a neighbor of $u$. In particular, $T$ is a dominating tree of $G$. 
\item\label{T2} If $r(x)$ and $r(y)$ are not related, then $xy \notin E(G)$.
\end{enumerate}
\end{lemma}

R\"odl~\cite{Rodl1986} proved the following theorem.
Its weaker version was later proved by Fox and Sudakov~\cite{FS2008b} without using the regularity lemma.
A set $U$ of vertices of $G$ 
is an \emph{$\varepsilon$-stable set} of a graph $G$
if $G[U]$ has at most $\varepsilon\binom{\abs{U}}{2}$ edges.
Similarly, $U$ is an \emph{$\varepsilon$-clique} of a graph $G$
if $G[U]$ has at least $(1-\varepsilon)\binom{\abs{U}}{2}$ edges.

\begin{theorem}[R\"odl~\cite{Rodl1986}]\label{thm: rodl}
For all $\epsilon>0$ and a graph $H$, 
there exists $\delta>0$ such that 
every $n$-vertex graph $G$ with no induced subgraph isomorphic to $H$
has an $\varepsilon$-stable set or an $\varepsilon$-clique of size at least $\delta n$.
\end{theorem}

We will use the following simple lemma. We present its proof for completeness.
\begin{lemma}\label{lem:bounded}
    Let $G$ be a graph. 
            Every  $\varepsilon$-stable set $U$ of $G$
            has a subset $U'$ of size at least $\abs{U}/2$
            with $\Delta(G[U'])\le 4\varepsilon \abs{U'}$.

\end{lemma}
\begin{proof}
Let $U'$ be the set of vertices of degree at most 
    $2\varepsilon \abs{U}$ in $G[U]$.
    Because 
    $\sum_{v\in U} \deg_{G[U]}(v) < \varepsilon \abs{U}^2$, we have $\abs{U'}\ge \abs{U}/2$. Moreover, for each vertex $v\in U'$, we have $\deg_{G[U']} (v)\le 2\varepsilon \abs{U}\le 4\varepsilon\abs{U'}$.
\end{proof}

Using Lemma~\ref{lem:bounded}, we can deduce the following corollary of Theorem~\ref{thm: rodl}.
\begin{corollary}\label{cor:rodl}
    For all $\alpha>0$ and a graph $H$, 
there exists $\delta>0$ such that 
every graph $G$ with no induced subgraph isomorphic to $H$
has a set $U\subseteq V(G)$ with $\abs{U}\ge \delta \abs{V(G)}$
such that 
either $\Delta(G[U])\le \alpha \abs{U}$
or $\Delta(\overline{G}[U])\le \alpha\abs{U}$.
\end{corollary}

The following easy lemma will be used to find a connected induced subgraph inside the output of Corollary~\ref{cor:rodl}.
We omit its easy proof.
\begin{lemma}\label{lem:conn}
    A graph $G$ has a pure pair $(A,B)$ such that $\abs{A}, \abs{B}\ge \abs{V(G)}/3$
    or 
    has a connected induced subgraph $H$ such that 
    $\abs{V(H)}\ge \abs{V(G)}/3$.
\end{lemma}

\begin{lemma}[Bonamy, Bousquet, and Thomass\'e~{\cite[Lemma 3]{BBT}}]\label{lem: tree path weight}
Let $T$ be a tree rooted at $v_r$ and $w: V(T) \rightarrow \mathbb{R}$ be a non-negative weight function on $V(T)$ with $\sum_{x\in V(T)}w(x)=1$.
Then there exists either a path $P$ from $v_r$ with weight at least $1/4$ or two unrelated sets $A$ and $B$ both with weight at least $1/4$.
\end{lemma}

A \emph{hole} is an induced cycle of length at least $5$.
\begin{lemma}[Bonamy, Bousquet, and Thomass\'e~{\cite[Lemma 4]{BBT}}]\label{lem: one direction}
    For given $k\ge 3$, there exist $\alpha=\alpha(k)>0$ and $\epsilon=\epsilon(k)>0$ such that   
       for any $n$-vertex graph $G$ with $n\geq 2$
    and $\Delta(G)\leq \alpha n$, 
    if $G$ has no holes of length at least $k$ and has a dominating induced path, then $G$ contains a pair $(A,B)$ of disjoint vertex sets 
    such that $A$ is anticomplete to $B$ and 
    $\abs{A},\abs{B}\ge \epsilon n$.
\end{lemma}

\section{Finding a cycle as a pivot-minor}\label{sec:cycle}

\begin{figure}
    \begin{tikzpicture}
        \tikzstyle{every node}=[circle,draw,fill=black!50,inner sep=0pt,minimum width=4pt]
        \node at(-2,1.5) [label=$v$](v){};
        \node at(2,1.5) [label=$u$](w){};
        \foreach \x in {0,...,5}{
            \node at (36*\x:1) (x\x){};
        }
        \foreach \x in {0,...,3}{
            \draw (w)--(x\x);
        }
        \foreach \x in {2,...,5}{
            \draw (v)--(x\x);
        }
        \draw (v)--(w);
        \draw (x0)--(x1);
        \draw (x4)--(x5);
        \draw [thick](x1)--(x5)--(x0);
        \draw [thick](x4)--(x1);
        \draw [thick](x2)--(x5);
        \node at (0,-.5) [rectangle,draw=none,fill=none]{$G$};
        \begin{scope}[xshift=5cm]
            \node at (0,-.5) [rectangle,draw=none,fill=none]{$G\pivot uv$};
            \node at(-2,1.5) [label=$u$](v){};
            \node at(2,1.5) [label=$v$](w){};
            \foreach \x in {0,...,5}{
                \node at (36*\x:1) (x\x){};
            }
            \foreach \x in {0,...,3}{
                \draw (w)--(x\x);
            }
            \foreach \x in {2,...,5}{
                \draw (v)--(x\x);
            }
            \draw (v)--(w);
            \draw (x0)--(x1);
            \draw (x4)--(x5);
            \draw [color=red,thick] (x4)--(x0);
            \draw [color=blue,thick] (x5)--(x3)--(x4)--(x2);
            \draw [color=green,thick] (x0)--(x2)--(x1)--(x3)--(x0);
        \end{scope}
    \end{tikzpicture}
    \caption{Pivoting $uv$.}\label{fig:pivot}
\end{figure}
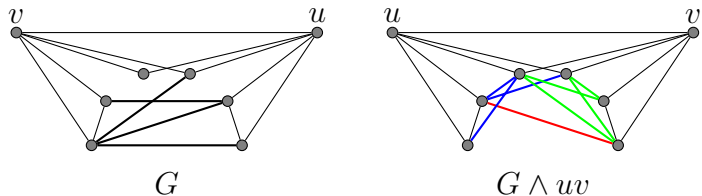
For a given graph $G$ and an edge $uv$, a graph $G \pivot uv$ obtained from $G$ by \emph{pivoting} $uv$ is defined as follows. 
Let $V_1= N_G(u)\cap N_G(v)$, $V_2=N_G(u)-N_G(v)$, $V_3=N_G(v)-N_G(u)$.
Then $G\pivot uv$ is the graph obtained from $G$ by complementing adjacency between vertices between $V_i$ and $V_j$ for all $1\leq i< j\leq 3$
and swapping the label of $u$ and $v$. 
See Figure~\ref{fig:pivot} for an illustration.
We say that $H$ is a \emph{pivot-minor} of $G$ if $H$ can be obtained from $G$ by deleting vertices and pivoting edges.
For this paper, we will also say that $H$ is a pivot-minor of $G$, when $G$ has a pivot-minor isomorphic to $H$.
A pivot-minor $H$ of $G$ is \emph{proper} if $\abs{V(H)}<\abs{V(G)}$.

We describe several scenarios for constructing $C_k$ as a pivot-minor. 
The following proposition is an easy one; One can obtain a desired pivot-minor from a longer cycle of the same parity.
\begin{proposition}\label{prop:cycle1}
For $m\ge k\ge 3$ with $m\equiv k\pmod 2$, the cycle $C_m$ has a pivot-minor isomorphic to $C_k$.
\end{proposition}
\begin{proof}
    We proceed by induction on $m-k$. We may assume that $m>k$. 
    Let $xy$ be an edge of $C_m$. Then $(C_m\pivot xy)-x-y$ is isomorphic to $C_{m-2}$, which contains a pivot-minor isomorphic to $C_k$ by the induction hypothesis.
\end{proof}
\begin{proposition}\label{prop:antihole}
    For integers $k\ge 3$ and $m\ge \frac{3}{2} k+6$, the graph $\overline{C_m}$  has a pivot-minor isomorphic to $C_k$.
\end{proposition}

Before proving Proposition~\ref{prop:antihole}, 
we present a simple lemma on partial complements of the cycle graph.
The \emph{partial complement}\footnote{We found this concept in a paper by Kami\'nski, Lozin, and 
Milani\v{c}~\cite{KLM2009}, though it may have been studied previously, as it is a  natural concept.} $G\oplus S$
of a graph $G$ by a set $S$ of vertices
is a graph obtained from $G$ by changing all edges within $S$ to non-edges and non-edges within $S$ to edges.

For $s\ge t\geq 0$, we say that $G$ is an \emph{$(s,t)$-cycle}
if $G$ is isomorphic to a graph $C_s\oplus X$
for a set $X$ of $t$ consecutive vertices in the cycle $C_s$.

\begin{figure}
    \begin{tikzpicture}[baseline]
        \tikzstyle{every node}=[circle,draw,fill=black!50,inner sep=0pt,minimum width=4pt]
        \fill [fill opacity=.1,fill=black,rounded corners] (-3,-3) rectangle  ++(1.5,6);
        \foreach \x in {1,...,3} {
            \node at (150-\x*30:2) [label=$v_{\x}$] (v\x) {};
        }
        \foreach \x in {4,...,6} {
            \node at (150-\x*30:2) [label=right:$v_{\x}$] (v\x) {};
        }   
        \foreach \x in {7,...,9} {
            \node at (150-\x*30:2) [label=below:$v_{\x}$] (v\x) {};
        }                  
        \node at (150:2) (v0){};
        \node at (-150:2) (v10){};
        \draw (v0)--(v1);
        \draw (v9)--(v10);
        \draw [snake=coil](v0)--(v10);
        \begin{scope}[dashed]
            \draw (v1)--(v2)--(v3)--(v4)--(v5)--(v6)--(v7)--(v8)--(v9);
        \end{scope}
        \foreach \x in {1,...,7}{
            \pgfmathtruncatemacro{\ystart}{\x+2}
            \foreach \y in {\ystart,...,9}{
                \draw (v\x)--(v\y);
            }
        }
        \draw [thick] (v2)--(v8);
    \end{tikzpicture} 
    $\quad\Rightarrow\quad$ 
        \begin{tikzpicture}[baseline]
            \tikzstyle{every node}=[circle,draw,fill=black!50,inner sep=0pt,minimum width=4pt]
            \fill [fill opacity=.1,fill=black,rounded corners] (-3,-3) rectangle  ++(4.3,6);
            \foreach \x in {1,3} {
                \node at (150-\x*30:2) [label=$v_{\x}$] (v\x) {};
            }
            \foreach \x in {4,...,6} {
                \node at (150-\x*30:2) [label=right:$v_{\x}$] (v\x) {};
            }   
            \foreach \x in {7,9} {
                \node at (150-\x*30:2) [label=below:$v_{\x}$] (v\x) {};
            }                  
            \node at (150:2) (v0){};
            \node at (-150:2) (v10){};
            \draw (v0)--(v1)--(v3)--(v4);
            \draw (v6)--(v7)--(v9)--(v10);
            \begin{scope}[dashed]
                \draw (v4)--(v5)--(v6);
                \draw (v1)--(v9)--(v3)--(v7)--(v1);
                \draw (v1)--(v4)--(v9);
                \foreach \x in {4,5,6}{
                    \draw (v1)--(v\x)--(v9);
                }
                \draw (v5)--(v7)--(v4);
                \draw (v6)--(v3)--(v5);
            \end{scope}
            \draw (v4)--(v6);
            \draw [snake=coil](v0)--(v10);
    \end{tikzpicture}  
    \caption{Obtaining an $(s-2,3)$-cycle from an $(s,9)$-cycle when $s>9$ in the proof of Lemma~\ref{lem:partialcomplement}.}\label{fig:reduce}
\end{figure}
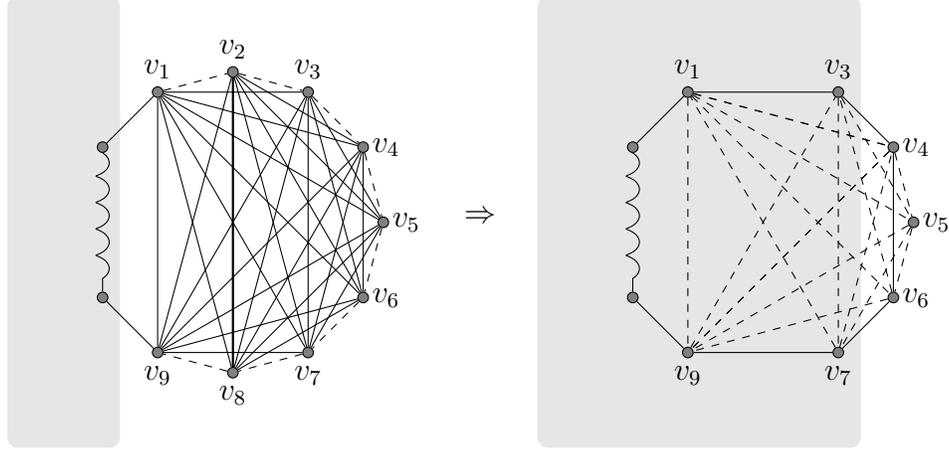
\begin{lemma}\label{lem:partialcomplement}
    Let     $s\ge t\ge 6$.
    An $(s,t)$-cycle contains a pivot-minor
    isomorphic to an $(s-2,t-6)$-cycle.
\end{lemma}
\begin{proof}
Let $v_1$, $\ldots$, $v_s$ be the vertices of $C_s$ in the cyclic order where $X=\{v_1,\dots, v_t\}$.
Then it is easy to check that $(C_s \oplus X) \pivot v_2v_{t-1} -\{v_2, v_{t-1}\}$ is isomorphic to $C_{s-2} \oplus X'$ where $X$ consists of $t-6$ consecutive vertices on the cycle. See Figure~\ref{fig:reduce}.
\end{proof}

\begin{proof}[Proof of Proposition~\ref{prop:antihole}] 
As $\overline{C_m}$ is an $(m,m)$-cycle, by Lemma~\ref{lem:partialcomplement}, $\overline{C_m}$ contains a pivot-minor isomorphic to an $(m-2i,m-6i)$-cycle for all $ i \le m/6$.

    Let us fix $i=\lceil (k-2)/4\rceil $. 
    Then 
    $  
            m-6i \ge m-6\cdot (k+1)/4\ge 9/2
    $  
    and therefore $\overline{C_m}$ contains a pivot-minor $H$
    isomorphic to an $(m-2i,m-6i)$-cycle and $m-6i\ge 5$.
    We may assume that $H=C_{m-2i}\oplus X$ where $C_{m-2i}=v_1\cdots v_{m-2i}$ and $X=\{v_{4i+1},\dots,v_{m-2i}\}$.
    
    Note that $H$ contains an induced cycle $C = v_{1}\cdots v_{4i} v_{4i+1} v_{m-2i} v_1$ of length $4i+2\ge k$.
    If $k$ is even, then by Proposition~\ref{prop:cycle1}, $H$ contains a pivot-minor 
    isomorphic to $C_k$.
    So we may assume that $k$ is odd
    and therefore $\abs{V(C)}=4i+2\ge k+1$.
    
    Let $x= v_{m-2i}$, $y=v_{4i+1}$ be the two vertices in $V(C)\cap X$. Since $m-6i\ge 5$, there is a 
    common neighbor $z$ of $x$ and $y$ in $X$.
    Then $z$ has exactly two neighbors $x$ and $y$ in $V(C)$.
    Then $H[V(C)\cup \{z\}] \pivot yz -y-z$ is a cycle of length $4i+1$.
    Since  $4i+1\ge k$,
    by Proposition~\ref{prop:cycle1}, it contains a pivot-minor isomorphic to $C_k$.
\end{proof}

A \emph{generalized fan} is a graph~$G$ with a specified vertex $c$, called the \emph{center},
such that $G-c$ is an induced path of length at least $1$, called the \emph{main path} of $G$
and both ends of the main path are adjacent to $c$.
If $c$ is adjacent to all vertices of $G-c$, then $G$ is called a \emph{fan}. 

An \emph{interval} of a generalized fan with a center $c$ 
is a maximal subpath of the main path having no internal vertex adjacent to $c$. 
The \emph{length} of an interval is its number of edges.
A generalized fan is an \emph{$(a_1,\dots, a_s)$-fan} if the lengths of intervals are $a_1,\dots, a_s$ in order. Note that an $(a_1,\dots, a_s)$-fan is also an $(a_s,\dots, a_1)$-fan.
An $(a_1,\dots, a_s)$-fan is a \emph{$k$-good fan} if $a_1\geq k-2$ or $a_s\ge k-2$.
An $(a_1,\dots, a_s)$-fan is a \emph{strongly $k$-good fan} if 
$s\ge 2$ and 
either $a_1\ge k-2$ and $a_s$ is odd,
or $a_s\ge k-2$ and $a_1$ is odd.
It is easy to observe that every $k$-good fan has a hole of length at least $k$.
However, that does not necessarily lead to a pivot-minor isomorphic to $C_k$
because of the parity issues.
In the next proposition, we show that every strongly $k$-good fan has a pivot-minor isomorphic to $C_k$.

\begin{proposition}\label{prop:fan}
    Let $k\ge 5$ be an integer.
    Every strongly $k$-good fan has a pivot-minor isomorphic to $C_k$.
\end{proposition}
\begin{proof}
    Let $G$ be an $(a_1,\dots, a_s)$-fan such that $s\ge 2$, $a_1\ge k-2$, and $a_s$ is odd.
    We proceed by the induction on $\abs{V(G)}$.
    We may assume that $G$ has no proper pivot-minor that is a strongly $k$-good fan.
    Note that $C_{a_1+2}$ is an induced subgraph of $G$, hence if $a_1\equiv k\pmod 2$, then $C_k$ is isomorphic to a pivot-minor of $G$ by Proposition~\ref{prop:cycle1}.
    Thus we may assume that $a_1\not\equiv k\pmod 2$ and so $a_1\geq k-1$.

    If $a_i$ is odd for some $1<i<s$, then $G$ contains a smaller strongly $k$-good fan by taking the first $i$ intervals, contradicting our assumption. Thus $a_i$ is even for all $1<i<s$.
    If $a_i\geq 3$ for some $i>1$, then let $uv$ be an internal edge of the $i$-th interval. Then $G\pivot uv-u-v$ is a strongly $k$-good fan, contradicting our assumption. 
    Thus, we may assume that $a_i\le 2$ for all $i>1$ and so 
    $G$ is an $(a_1,2,\dots,2,1)$-fan.

    Let $xy$ be the last interval of $G$ with length $1$.
    Then $G\pivot xy-x-y$ is a $(a_1,2,\dots,2,1)$-fan with $s-1$ intervals. 
    By the assumption, we may assume that $s=2$ and $G\pivot xy-x-y$ is an $(a_1-1)$-fan with one interval, which is a cycle with $a_1+1 $ edges. As $a_1+1\ge k$ and $a_1+1\equiv k\pmod 2$, Proposition~\ref{prop:cycle1} implies that $G$ contains a pivot-minor isomorphic to~$C_k$. 
\end{proof}

\section{Proof of Theorem~\ref{thm: Strong EH}}\label{sec:proof}

First we choose $\alpha>0$ and $\epsilon_0>0$ so that 
\begin{equation}\label{alpha}
    \begin{minipage}[c]{0.9\textwidth}
    $4\alpha\leq \alpha(\lceil\frac{3}{2}k+6\rceil)$ and $\epsilon_0=\epsilon(\lceil \frac{3}{2}k+6\rceil)$ where $\alpha(\cdot),\epsilon(\cdot)$ are specified in Lemma~\ref{lem: one direction} 
    \end{minipage}
\end{equation}
and in addition $\alpha<1/(8k)$ as well.
Let $\delta>0$ be a constant obtained by applying Corollary~\ref{cor:rodl} with $\alpha/3$ as $\alpha$ and $C_k$ as $H$.
Choose $\epsilon>0$ so that 
\[
    \epsilon < \min\left(
    \frac{\delta}{12},
    (1-4(k+3)\alpha)\frac{\delta}{240},
    \frac{\epsilon_0\delta}{12}\right).
\]

Let $n>1$ be an integer and $G$ be an $n$-vertex graph with no pivot-minor isomorphic to~$C_k$. In particular, $G$ does not have $C_k$ as  an induced subgraph.
To derive a contradiction, we assume that
$G$ contains no pure pair $(A,B)$ with $\abs{A},\abs{B}\ge \epsilon n$.
We may assume that $\epsilon n>1$, because otherwise an edge or a non-edge of $G$ gives 
a pure pair. 

By Corollary~\ref{cor:rodl}, 
there exists a subset $U$ of $V(G)$ 
such that $\abs{U}\ge \delta \abs{V(G)}$
and $\Delta(G^0)[U])\le (\alpha/3) \abs{U}$
for some $G^0\in\{G,\overline G\}$.
By the assumption on $G$, 
$G^0[U]$ has no pure pair $(A,B)$ with $\abs{A},\abs{B}\ge (\epsilon/\delta)\abs{U}$.
As $\epsilon/\delta<1/3$, by Lemma~\ref{lem:conn}, 
$G^0[U]$ has a connected induced subgraph $G'$ such that
$\abs{V(G')}\ge \abs{U}/3$.
Let $n'=\abs{V(G')}$.

Then $n'\ge (\delta/3)n$ and $\Delta(G')\le (\alpha/3) \abs{U}\le \alpha n'$.
By the assumption on $G$, 
\begin{equation}\label{eq: basic assump2}
    \text{$G'$ contains no 
    pure pair $(A,B)$
    with $\abs{A},\abs{B}\ge (3\epsilon/\delta)n'$.
    }
\end{equation}

By applying Lemma~\ref{lem: tree} with $G'$, we obtain a dominating induced tree $T$ and $r:V(G')\rightarrow V(T)$ satisfying Lemma~\ref{lem: tree}~(T1)--(T2) with $G'$.
For each $u \in V(T)$, let 
\[w(u):=\frac{\abs{r^{-1}(\{u\})}}{n'}\] be the \emph{weight} of $u$. 
By applying Lemma~\ref{lem: tree path weight} with the weight $w$, we obtain either an induced path $P$ of $T$ with weight at least $1/4$ or two unrelated sets $A$ and $B$ both with weight at least $1/4$.

In the latter case, Lemma~\ref{lem: tree}~(T2) implies that
$r^{-1}(A)$ is anticomplete to $r^{-1}(B)$ in $G'$
and $\abs{r^{-1}(A)}, \abs{r^{-1}(B)}\ge n'/4\ge (3\epsilon/\delta)n'$, 
contradicting~\eqref{eq: basic assump2}.

Hence, there exists an induced path $P$ in $G'$ with $|V(P)\cup N_{G'}(V(P))| \geq n'/4$. Let $W:=V(P)\cup N_{G'}(V(P))$. 
Note that $n'/4\ge \delta n / 12> \epsilon n > 1$ and so $\abs{W}\ge 2$.

Suppose that $G'$ is an induced subgraph of $\overline{G}$. 
Using \eqref{alpha},  we apply Lemma~\ref{lem: one direction} to $G'[W]$ with $4\alpha$ as $\alpha$ and $\lceil \frac{3}{2}k+6 \rceil$ as $k$.
Then we can deduce from~\eqref{eq: basic assump2}
and $\epsilon' \abs{W} \ge \frac{12\epsilon}{\delta}\frac{n'}4=(3\epsilon/\delta)n'$
that  
the graph $G'[W]$ contains an induced cycle $C_m$  with $m\geq \lceil \frac{3}{2}k+6\rceil$
and by Proposition~\ref{prop:antihole}, $\overline{G'}$ contains a pivot-minor isomorphic to $C_k$, and so does $G$, a contradiction. 

Thus $G'$ is an induced subgraph of $G$. 
Let $G^*:= G'[W]$ and let $n^*=\abs{W}$.
Then $G^*$ has no pivot-minor isomorphic to $C_k$, $n^*\ge n'/4$,
and 
$\Delta(G^*)\le 4\alpha n^*$.
By~\eqref{eq: basic assump2}, $G^*$ contains no 
pure pair $(A,B)$ with $\abs{A},\abs{B}\ge (12\epsilon/\delta)n^*$.
Now the theorem follows from applying the 
following lemma with $G^*, n^*, 4\alpha, 12\epsilon/\delta$ playing the roles of $G, n, \alpha, \epsilon$ respectively in the statement of the lemma.

\begin{lemma}\label{lem: main}
    Let $k\ge 3$ be an integer.
    Let $0<\alpha<1/(2k)$, 
    $0<\epsilon\le (1-(k+3)\alpha)/20$.
    Let $G$ be a graph on $n\ge 2$ vertices such that
    $\Delta(G)\le \alpha n$
    and 
    $G$ has no pure pair $(A,B)$ with 
    $\abs{A},\abs{B}\ge \epsilon n$.
    If $G$ has a dominating induced path $P$, then 
    $G$ has a pivot-minor isomorphic to $C_k$.
\end{lemma}
\begin{proof}
    Suppose that $G$ has no pivot-minor isomorphic to $C_k$. 
    Note that $\epsilon n>1$ as otherwise we have a pure pair on two vertices since $n\ge 2$.
    Let us label vertices of $P$ by $1$, $2$, $\ldots$, $s$ in the order.

    As $P$ is a dominating path of $G$ and $1\le \Delta(G) \leq \alpha n$, we have 
    $2\alpha n s\ge (\alpha n+1) s \ge n$ and therefore 
    \[
        s\ge 1/(2\alpha).
    \]
    Note that $s-k>0$ because $\alpha<\frac{1}{2k}$.
    As $P$ is an induced path, it contains 
    a pure pair $(A,B)$ with $\abs{A},\abs{B}\ge \lfloor \frac{s-1}{2}\rfloor$ and so
    $\frac{s-2}2\le \lfloor \frac{s-1}{2}\rfloor<\epsilon n$.
    Because $\epsilon n>1$, we have $2\epsilon n +2< 4\epsilon  n $ and so 
    \begin{equation}\label{eq: s size}
        s < 2 \epsilon  n  +2  < 4\epsilon  n .
    \end{equation}

Now, for each $i\in [s-k+1]$, let
\[
U_i^{-} :=\{ 1,\dots, i-1\}, \enspace U_i^0:= \{ i,\dots, i+k-1\}, \enspace \text{and} \enspace
U_i^+ := \{i+k,\dots, s\}.
\]
In other words, this partitions $P$ into three (possibly empty) subpaths.
Furthermore, for all $i\in [s-k+1]$ and $u\in N_G(U_{i}^-)-V(P)$, let
\[
    m_i^-(u):= \max ( N_{G }(u)\cap U_i^-) 
\]
and for all  $i\in [s-k+1]$ and $u\in N_G(U_{i}^+)-V(P)$, let 
\[
    m_i^+(u):= \min ( N_{G }(u)\cap U_i^+),
\]
indicating the largest neighbor of $u$ in $U^-_i$ and the smallest neighbor of $u$ in $U^+_i$ respectively.
For each $i\in [s-k+1]$, let
\begin{align*}
    A_i &:= N_{G}( U_i^0)- V(P) \text{ and }  \\
    B_i & := (N_G(U_i^-)\cap N_G(U_i^+))- (A_i\cup V(P)).
\end{align*}
Note that for each $u\in B_i$, we have
\begin{equation}\label{eq: B parity}
m^+_i(u) - m^-_i(u) \not\equiv k\pmod 2,
\end{equation}
because otherwise $(u, m^-_i(u),m^-_i(u)+1,\dots, m^+_i(u),u)$ forms an induced cycle of length at least $k$ and 
Proposition~\ref{prop:cycle1} implies that $G$ contains a pivot-minor isomorphic to $C_k$, a contradiction.

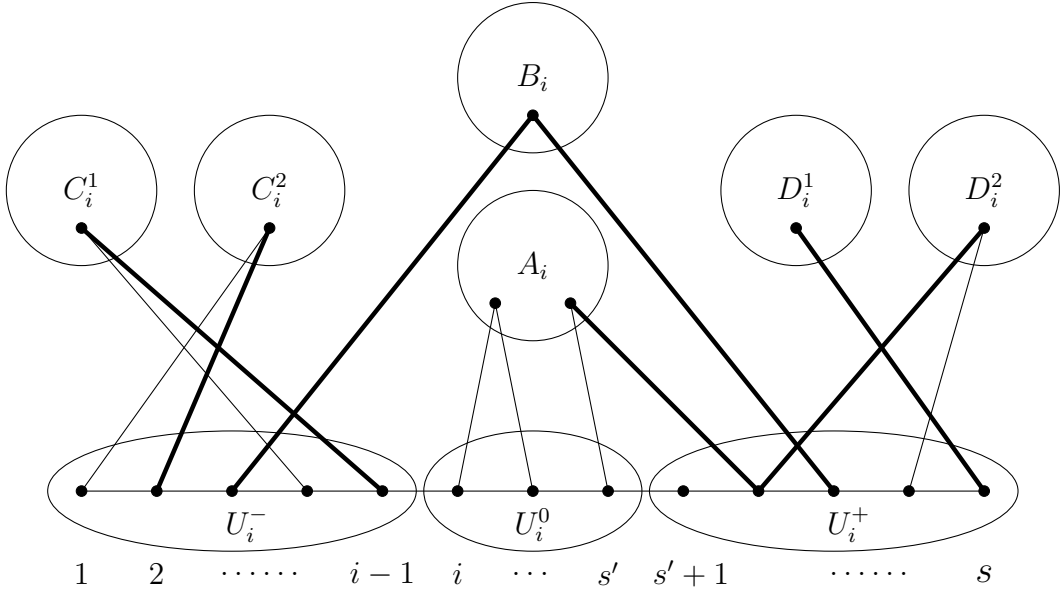
\begin{figure}\label{fig:1}
\centering
\begin{tikzpicture}

\draw[fill=none] (3,0) ellipse [x radius=2.45,y radius=0.8];

\node at (3.2,-0.45) { $U_i^-$};

\draw[fill=none] (7,0) ellipse [x radius=1.45,y radius=0.8];

\node at (7,-0.45) { $U_i^0$};

\draw[fill=none] (11,0) ellipse [x radius=2.45,y radius=0.8];

\node at (11.2,-0.45) { $U_i^+$};

\draw[fill=none] (7,3) ellipse [x radius=1,y radius=1];
\node at (7,3) { $A_i$};
\filldraw[fill=black] (7.5, 2.5) circle (2pt);
\draw (7.5,2.5) to (8,0);
\draw [line width=0.6mm](7.5,2.5) to (10,0);
\filldraw[fill=black] (6.5, 2.5) circle (2pt);
\draw (6.5,2.5) to (7,0);
\draw (6.5,2.5) to (6,0);

\draw[fill=none] (7,5.5) ellipse [x radius=1,y radius=1];
\node at (7,5.5) { $B_i$};
\filldraw[fill=black] (7, 5) circle (2pt);
\draw [line width=0.6mm](7, 5) to (11,0);
\draw [line width=0.6mm](7, 5) to (3,0);

\draw[fill=none] (1,4) ellipse [x radius=1,y radius=1];
\node at (1,4) { $C^1_i$};

\draw[fill=none] (3.5,4) ellipse [x radius=1,y radius=1];
\node at (3.5,4) { $C^2_i$};

\filldraw[fill=black] (1, 3.5) circle (2pt);
\draw[line width=0.6mm] (1, 3.5) to (5,0);
\draw (1, 3.5) to (4,0);

\filldraw[fill=black] (3.5, 3.5) circle (2pt);
\draw (3.5, 3.5) to (1,0);
\draw[line width=0.6mm] (3.5, 3.5) to (2,0);

\draw[fill=none] (10.5,4) ellipse [x radius=1,y radius=1];
\node at (10.5,4) { $D^1_i$};

\draw[fill=none] (13,4) ellipse [x radius=1,y radius=1];
\node at (13,4) { $D^2_i$};

\filldraw[fill=black] (10.5, 3.5) circle (2pt);
\draw[line width=0.6mm] (10.5, 3.5) to (13,0);

\filldraw[fill=black] (13, 3.5) circle (2pt);
\draw[line width=0.6mm] (13, 3.5) to (10,0);
\draw (13, 3.5) to (12,0);

\filldraw[fill=black] (1,0) circle (2pt);
\filldraw[fill=black] (2,0) circle (2pt);
\filldraw[fill=black] (3,0) circle (2pt);
\filldraw[fill=black] (4,0) circle (2pt);
\filldraw[fill=black] (5,0) circle (2pt);
\filldraw[fill=black] (6,0) circle (2pt);
\filldraw[fill=black] (7,0) circle (2pt);
\filldraw[fill=black] (8,0) circle (2pt);
\filldraw[fill=black] (9,0) circle (2pt);
\filldraw[fill=black] (10,0) circle (2pt);
\filldraw[fill=black] (11,0) circle (2pt);
\filldraw[fill=black] (12,0) circle (2pt);
\filldraw[fill=black] (13,0) circle (2pt);

\draw (1,0) to (13,0);

\node at (1,-1.1) { $1$};
\node at (2,-1.1) { $2$};
\node at (3.4,-1.1) { $\dots\dots$};
\node at (5,-1.1) { $i-1$};
\node at (6,-1.1) { $i$};
\node at (7,-1.1) { $\dots$};
\node at (8,-1.1) { $s'$};
\node at (9.1,-1.1) { $s'+1$};
\node at (11.5,-1.1) { $\dots\dots$};

\node at (13,-1.1) {\large $s$};

\end{tikzpicture}
\caption{$s'=i+k-1$. Bold lines indicate $m_i^-(u)$ and $m_i^+(u)$.}
\end{figure}

For each $i\in [s-k+1]$, let
\begin{align*}
C_i^{1} &:= \{ u\in N_{G}(U_i^{-})- (A_i\cup B_i\cup V(P)) : m_i^-(u)\equiv1\pmod2\}, \\
C_i^{2} &:= \{ u\in N_{G}(U_i^{-})- (A_i\cup B_i\cup V(P)) : m_i^-(u)\equiv0\pmod2\}, \\
D_i^{1} &:= \{ u\in N_{G}(U_i^{+})- (A_i\cup B_i\cup V(P)) : m_i^+(u) \equiv k \pmod2\}, \enspace \text{and} \\
D_i^{2} &:= \{ u\in N_{G}(U_i^{+})- (A_i\cup B_i\cup V(P)) : m_i^+(u) \equiv k +1\pmod2\}.
\end{align*}
Recall that $P$ is dominating. Hence, for each $i$, the sets $\{A_i, B_i, C_i^1, C_i^2, D_i^1, D_i^2, V(P)\}$ forms a partition of $V(G)$ into $7$ possibly empty sets.

If there exists an edge between $u\in C_i^j$ and $v\in D_i^j$ for some $j\in [2]$, then we obtain an induced cycle $(u, m_i^-(u), m_i^-(u)+1,\dots, m_i^+(v), v,u)$ having length $m_i^+(v)-m_i^-(u)+3>k$ and 
$m_i^+(v)-m_i^-(u)+3\equiv k \pmod 2$,
contradicting our assumption that $G$ has no pivot-minor isomorphic to $C_k$ by Proposition~\ref{prop:cycle1}.
Thus $C_i^j$ is anticomplete to $D_i^j$. Hence, 
\begin{equation}\label{eq: min CD}
    \min\{ |C_i^{j}|, |D_i^{j}| \} <\epsilon n.
\end{equation}
for all $i\in [s-k+1]$ and $j\in[2]$. 
Furthermore, we prove the following.
\begin{claim}\label{cl:parity}
    Let $i\in [s-k+1]$.
    For each $v\in B_i$, 
    all integers in $N_G(v)\cap U_{i}^-$ have the same parity
    and all integers in $N_G(v)\cap U_{i}^+$ have the same parity. 
\end{claim}
\begin{subproof}[Proof of Claim~\ref{cl:parity}]
    If $N_G(v)\cap U_{i}^+$ has two integers $a<b$ of the different parity, then $G$ contains a strongly $k$-good generalized fan by taking a subpath of $P$ from $m_i^-(v)$ to $b$ as its main path and $v$ as its center. 
    Then by Proposition~\ref{prop:fan},
    $G$ contains a pivot-minor isomorphic to $C_k$, contradicting the assumption.
    Thus all integers in $N_G(v)\cap U_i^+$ have the same parity and similarly all integers in $N_G(v)\cap U_i^-$ have the same parity.
\end{subproof}
\begin{claim}\label{cl:b}
    For all $i\in [s-k+1]$, 
    $|B_i|<2(\alpha +2\epsilon )n$.
\end{claim}
\begin{subproof}[Proof of Claim~\ref{cl:b}]
Suppose $|B_i|\geq 2(\alpha +2\epsilon ) n $ for some $i\in[s-k+1]$.
Then there exists $r_B\in\{0,1\}$ such that 
\[
    B':=\{ u \in B_i : m^-_i (u)\equiv r_B\pmod2\}
\] 
has size at least $(\alpha +2\epsilon ) n $.
By~\eqref{eq: B parity}, $m^+_i(u)\equiv k+r_B+1\pmod 2$ for all $u\in B'$. 

We claim that if $uv$ is an edge in $G [B']$, then $(m^-_i(u), m^+_i(u)) = (m^-_i(v), m^+_i(v))$. 
Suppose not. Without loss of generality, we may assume that $m^-_i(u) < m_i^-(v)$, because otherwise we may reverse the ordering of $P$ to ensure that $m^-_i(u)\neq m^-(v)$ and swap $u$ and $v$ if necessary.

If %
$m^+_i(u)\ge m^+_i(v)$, 
then by Claim~\ref{cl:parity}, 
$\{m^-_i(v), m^-_i(v)+1,\dots, m^+_i(u),u,v\}$ induces a strongly $k$-good generalized fan with $v$ as a center
and $(m_i^-(v),m_i^-(v)+1,\dots, m_i^+(u),u)$ as its main path. 
This implies that $G$ has a pivot-minor isomorphic to $C_k$ by Proposition~\ref{prop:fan}, contradicting our assumption.

If $m^+_i(u)< m^+_i(v)$, then $(m^-_i(v),m^-_i(v)+1,\dots, m^+_i(u),u,v)$ is an induced cycle of length $m^+_i(u)-m^-_i(v) +3\ge k$, and 
$m^+_i(u)-m^-_i(v) +3\equiv (k+r_B+1)-r_B+3 \equiv k\pmod2$, a contradiction by Proposition~\ref{prop:cycle1}. 

\medskip
Hence, $(m^-_i(u), m^+_i(u))=(m^-_i(v), m^+_i(v))$ for all $uv\in E(G [B'])$.
Let $C_1$, $\ldots$, $C_t$ be the connected components of $G [B']$. By the above observation, for each $j\in [t]$, there exist $a_j \in U^-_i$ and $b_j\in U^+_i $ such that $V(C_j) \subseteq N_{G }(a_j)\cap N_{G }(b_j)$. So, $\abs{V(C_j)}\leq \alpha  n $.
As $|B'|\geq (\alpha +2\epsilon ) n $, 
there exists a set $I\subseteq\{1,2,\ldots,t\}$ such that 
$\epsilon  n \le \abs{\bigcup_{i\in I}V(C_i)}\le (\alpha +\epsilon )n $.
Let $A:=\bigcup_{i\in I}V(C_i)$ and $B:=B'- A$.
Then $(A,B)$ is a pure pair of $G$ with $\abs{A},\abs{B}\ge \epsilon  n $, a contradiction. 
\end{subproof}
\begin{claim}\label{cl: CC and DD large}
There exist $i_* \in [s-k+1]$ and $j_*\in [2]$ such that 
\[
    |C_{i_*}^{j_*}|, |D_{i_*}^{3-j_*}| \ge 3\epsilon n.
\]
\end{claim}
\begin{subproof}[Proof of Claim~\ref{cl: CC and DD large}]
First, since $\Delta(G )\le \alpha n $, 
$|A_i|\leq k\alpha  n $ for each $i\in [s-k+1]$.

Let $f(i):= \abs{C_i^1}+ \abs{C_i^2}$. Then 
\begin{align*}
    f(1)&=0,\\
    f(s-k+1)&=n - \abs{A_{s-k+1}} - s  \quad\text{because $U_{s-k+1}^+=D_{s-k+1}^1=D_{s-k+1}^2=B_{s-k+1}=\emptyset$,}
    \\
    &\ge n - k \alpha   n  - 4\epsilon  n  \quad\text{ by \eqref{eq: s size}  and the assumption that $\Delta(G)\leq \alpha n$, }\\
    &= (1-k\alpha -4\epsilon )n \ge 6\epsilon n,
\intertext{and for each $i\in [s-k]$, we have }
    f(i+1)-f(i) &\le \deg_{G }(i)\leq \alpha  n .
\end{align*}
Hence, there exists $i_*\in [s-k+1]$ such that $6\epsilon  n  \le f(i_*) < (6\epsilon +\alpha )n $. 
As $|B_{i_*}|< 2(\alpha +2\epsilon )n $, we have 
\begin{align*}
    \abs{D_{i_*}^1}+\abs{D_{i_*}^2}
    &=n-\abs{A_{i_*}}-\abs{B_{i_*}}
    -(\abs{C_{i_*}^1}+\abs{C_{i_*}^2})-\abs{V(P)}\\
    &\ge n - k\alpha  n  
    - 2(\alpha +2\epsilon )n 
    -(6\epsilon +\alpha )n 
    -4\epsilon n \\
    &= (1-(k+3)\alpha -14\epsilon )n 
    \ge 6\epsilon n .
\end{align*}
So, there exist $a, b\in \{1,2\}$ such that
$\abs{C_{i_* }^a}, \abs{D_{i_*}^b}\ge  3\epsilon  n $.
By~\eqref{eq: min CD}, $a\neq b$ and so we take $j_*:=a$. This proves the claim.
\end{subproof}

\begin{claim}\label{eq: complete}
    For each component $C$ of $G [C_{i_*}^{j_*}]$
    and each component $D$ of $G [D_{i_*}^{3-j_*}]$, 
    $(C, D)$ is a pure pair of $G $.
\end{claim}
\begin{subproof}[Proof of Claim~\ref{eq: complete}]
    Assume not. 
    By symmetry, we may assume that $C$ has a vertex $u$ having both a neighbor and a non-neighbor in $D$, because otherwise we swap $C$ and $D$ by reversing the order of $P$.
    As $D$ is connected, there exist $v,v'\in V(D)$ such that $uv,vv'\in E(G )$ and $uv'\notin E(G )$.

    Note that $m^+_{ i_*}(v)\equiv m^+_{ i_*}(v')\pmod 2$ and 
    \begin{equation}\label{eq: parity same}
        \text{\emph{
        for every neighbor $\ell \in N_{G }(v)\cap U^+_{ i_*}$, the number $\ell - m_{ i_*}^+(v)$ is even,
        }}
    \end{equation} 
    because otherwise for the minimum $\ell\in N_{G }(v)\cap U_{i^*}^+$ with odd $\ell - m_{ i_*}^+(v)$,
    a vertex set $\{v,m^-_{ i_*}(u), m^-_{ i_*}(u)+1,\dots, \ell,u\}$
    induces a strongly $k$-good generalized fan with $v$ as its center,
    a contradiction by Proposition~\ref{prop:fan}.

    If $m^+_{ i_*}(v) \leq m^+_{ i_*}(v')$, then 
    $\{v,u,m^-_{ i_*}(u),m^-_{ i_*}(u)+1,\dots, m^+_{i_*}(v'),v'\}$ induces a strongly $k$-good generalized fan with $v$ as a center by~\eqref{eq: parity same}.

    If $m^+_{ i_*}(v) > m^+_{ i_*}(v')$, then simply 
    $(u, m^-_{ i_*}(u), m^-_{ i_*}(u)+1,\dots, m^+_{ i_*}(v'), v', v,u )$ is an induced cycle 
    whose length is at least $k$ and is of the same parity with $k$. Hence Proposition~\ref{prop:cycle1} implies a contradiction. 
\end{subproof}

By Claim~\ref{eq: complete},
there exists $S\in \{C_{i_*}^{j_*},D_{i_*}^{3-j_*}\}$ such that 
every component of $G [S]$ has less than $\epsilon  n $ vertices.
By Claim~\ref{cl: CC and DD large},
we can greedily find a set of components of $G [S]$ covering 
at least $\epsilon  n $ vertices and at most $2\epsilon  n $ vertices.
Since $\abs{S}\ge 3\epsilon n $, 
the vertices of $S$ covered by this set of components
with the vertices of $S$ not covered by this set of components
give a pure pair $(A,B)$ with $\abs{A},\abs{B}\ge \epsilon  n $, a contradiction. 
This proves the lemma.
\end{proof}

\section{Discussions}\label{sec:discussion}
For a graph $G$, we write $\chi(G)$ to denote its chromatic number 
and $\omega(G)$ to denote its clique number, that is the maximum size of a clique.
A class $\mathcal G$ of graphs is called \emph{$\chi$-bounded} if 
there exists a function $f:\mathbb Z\to \mathbb Z$ such that
for every induced subgraph $H$ of a graph in $\mathcal G$, 
$\chi(H)\le f(\omega(H))$. 
In addition, we say $\mathcal G$ is \emph{polynomially $\chi$-bounded}
if $f$ can be taken as a polynomial.

Every polynomially $\chi$-bounded class of graphs 
has the strong Erd\H{o}s-Hajnal property, 
but the converse does not hold; 
see the survey paper by Scott and Seymour~\cite{SS2018}.
So it is natural to ask whether the class of graphs with no pivot-minor isomorphic to $C_k$ is polynomially $\chi$-bounded, which is still open. 
So far Choi, Kwon, and Oum~\cite{CKO2015} showed that 
it is $\chi$-bounded.
\begin{theorem}[Choi, Kwon, and Oum~{\cite[Theorem 4.1]{CKO2015}}]\label{thm:cko}
    For each $k\ge 3$, the class of graphs
    with no pivot-minor isomorphic to $C_k$ is
    $\chi$-bounded.
\end{theorem}
They showed that $\chi(G)\le 2(6k^3-26k^2+25k-1)^{\omega(G)-1}$ holds for graphs $G$ having no pivot-minor isomorphic to $C_k$, far from being a polynomial. 
Theorem~\ref{thm:cko} is now implied by a recent theorem 
of Scott and Seymour~\cite{SS2019}, solving three conjectures of Gy\'arf\'as~\cite{Gyarfas1987} on $\chi$-boundedness all at once. 
\begin{theorem}[Scott and Seymour~\cite{SS2019}]\label{thm:ss}
    For all $k\ge 0$ and $\ell>0$, 
    the class of all graphs having no induced cycle of length $k$ modulo $\ell$ is $\chi$-bounded.
\end{theorem}
To see why Theorem~\ref{thm:ss} implies Theorem~\ref{thm:cko}, take $\ell:=2\lceil k/2\rceil$ and apply Proposition~\ref{prop:cycle1}.
Still the bound obtained from Theorem~\ref{thm:ss} is far from being a polynomial. 

And yet no one was able to answer the following problem of Esperet.
\begin{problem}[Esperet; see~\cite{KM2016}]
    Is it true that every $\chi$-bounded class of graphs 
    polynomially $\chi$-bounded?
\end{problem}
Thus it is natural to pose the following conjecture.
\begin{conjecture}\label{con:polychi}
    For every graph $H$, the class of graphs with no pivot-minor isomorphic to $H$ is polynomially $\chi$-bounded.
\end{conjecture}
It is open whether Conjecture~\ref{con:polychi} holds when $H=C_k$.
Conjecture~\ref{con:polychi} implies  not only  Conjectures~\ref{con:pivot},~\ref{con:strongeh} but also the following 
conjecture of Geelen (see~\cite{DK2012}) proposed 
in 2009 at the DIMACS workshop on graph colouring and structure held at Princeton University.
\begin{conjecture}[Geelen; see~\cite{DK2012}]\label{con:geelen}
    For every graph $H$, the class of graphs with no vertex-minor isomorphic to $H$ is $\chi$-bounded.
\end{conjecture}
Of course it is natural to pose the following conjecture, weaker than Conjecture~\ref{con:polychi} but stronger than Conjecture~\ref{con:geelen}.
\begin{conjecture}[Kim, Kwon, Oum, and Sivaraman~\cite{KKOS2018}]\label{con:polychivertexminor}
    For every graph $H$, the class of graphs with no vertex-minor isomorphic to $H$ is polynomially $\chi$-bounded.
\end{conjecture}
For vertex-minors, more results are known.
Kim, Kwon, Oum, and Sivaraman~\cite{KKOS2018} proved that for each $k\ge 3$,
the class of graphs with no vertex-minor isomorphic to $C_k$ is polynomially $\chi$-bounded.
Their theorem is now implied by the following two recent theorems.
To describe these theorems, we first have to introduce a few terms.
A \emph{circle graph} is the intersection graph of chords in a circle. 
In particular, $C_k$ is a circle graph.
The \emph{rank-width} of a graph is one of the width parameters of graphs, measuring how easy it is to decompose a graph into a tree-like structure while keeping every cut to have a small `rank'. Rank-width was introduced by Oum and Seymour~\cite{OS2004}.
We will omit the definition of the rank-width. 
\begin{theorem}[Geelen, Kwon, McCarty, and Wollan~\cite{GKMW2019}]
    For each circle graph $H$, there is an integer $r(H)$ such that every graph 
    with no vertex-minor isomorphic to $H$ has rank-width at most $r(H)$.
\end{theorem}
\begin{theorem}[Bonamy and Pilipczuk~\cite{BP2019}]\label{thm:bonamy}
    For each $k$, the class of graphs of rank-width at most $k$ is polynomially $\chi$-bounded.
\end{theorem}

\begin{figure}
    \centering
    \begin{tikzcd}[column sep=small]
        & \text{Erd\H{o}s-Hajnal}\\
        \text{$\chi$-bounded}
        & \text{strong Erd\H{o}s-Hajnal}
        \arrow[u]\\
        \text{polynomially $\chi$-bounded}
        \arrow[u]\arrow[ru]\\
        \text{bounded rank-width}
        \arrow[u,"\text{\cite{BP2019}}"]
        \\
        && \text{no $C_k$-pivot-minors}
        \arrow[uuul]
        \arrow[uuull,"\text{\cite{CKO2015}}"]\\
        \text{no $C_k$-vertex-minors}
        \arrow[uu,"\text{\cite{GKMW2019}}"]\arrow[rru]        
    \end{tikzcd}
    \caption{Known implications between properties of classes of graphs.}\label{fig:relation}
\end{figure}
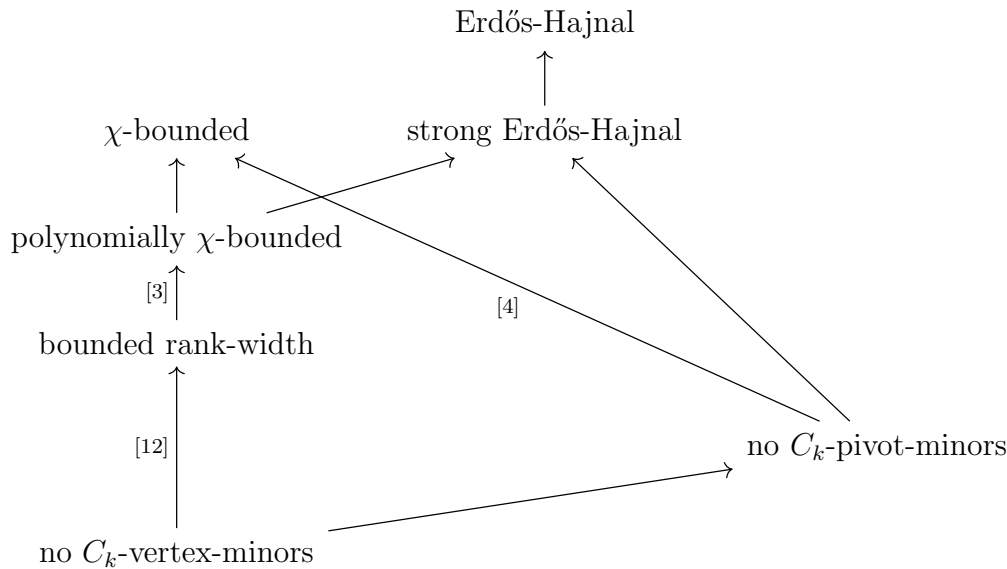

As noted in~\cite{CO2018}, it is easy to prove directly that the class of graphs of bounded rank-width has the strong Erd\H{o}s-Hajnal property, without using Theorem~\ref{thm:bonamy}.
See Figure~\ref{fig:relation} for a diagram showing the containment relations between these properties.

So, one may wonder whether the class of graphs with no pivot-minor isomorphic to $C_k$ has bounded rank-width.
Unfortunately, if $k$ is odd, then it is not true, 
because all bipartite graphs 
have no pivot-minor isomorphic to $C_k$ for odd $k$
and yet have unbounded rank-width, see~\cite{Oum2004}.
If $k$ is even, then it would be true if the following conjecture hold.
\begin{conjecture}[Oum~\cite{Oum2006a}]
    For every bipartite circle graph $H$, 
    there is an integer $r(H)$ such that every graph 
    with no pivot-minor isomorphic to $H$ has 
    rank-width at most $r(H)$.
\end{conjecture}

\subsection*{Note.} 
Chudnovsky, Scott, Seymour, and Spirkl~\cite{CSSS2018} 
proved that for every graph $H$, 
the class of graphs $G$ such that neither $G$ nor $\overline G$ has any subdivision of $H$ as an induced subgraph
has the strong Erd\H{o}s-Hajnal property.
This implies that
when $k$ is even, the class of graphs with no induced even hole of length at least $k$
and no induced even anti-hole of length at least $k$
has the strong Erd\H{o}s-Hajnal property.
This is because every subdivision of a large theta graph%
\footnote{A \emph{theta graph} is a graph consisting of three internally disjoint paths of length at least $1$ joining two fixed vertices.}  contains a large even hole.
This implies Theorem~\ref{thm:eh} for even $k$ but not for odd $k$
by Propositions~\ref{prop:cycle1} and \ref{prop:antihole}.
The authors would like to thank the authors of \cite{CSSS2018} to share this observation.

\subsection*{Acknowledgement}
The authors would like to thank anonymous reviewers for their careful reading and helpful suggestions. 

\providecommand{\bysame}{\leavevmode\hbox to3em{\hrulefill}\thinspace}
\providecommand{\MR}{\relax\ifhmode\unskip\space\fi MR }
\providecommand{\arxiv}{arXiv:}
\providecommand{\MRhref}[2]{%
  \href{http://www.ams.org/mathscinet-getitem?mr=#1}{#2}
}
\providecommand{\href}[2]{#2}


%

%
%
%
%
%
%
%

%

\end{document}